\documentclass[12pt]{amsart}

\usepackage{amssymb, amscd, txfonts}
\usepackage[all]{xy} 
\usepackage{graphicx}

\numberwithin{equation}{section}

\newtheorem{theorem}{Theorem}[section]
\newtheorem{proposition}[theorem]{Proposition}
\newtheorem{lemma}[theorem]{Lemma}
\newtheorem{corollary}[theorem]{Corollary}

\theoremstyle{definition}

\newtheorem{example}[theorem]{Example}

\theoremstyle{remark}
\newtheorem{remark}[theorem]{Remark}
\newtheorem{claim}[theorem]{Claim}

\newcommand{\Z}{\mathbb{Z}}
\newcommand{\Q}{\mathbb{Q}}

\newcommand{\C}{\mathbb{C}}

\newcommand{\proj}{{\mathbb P}}

\newcommand{\CH}{{\rm CH}}

\begin{document}

\title[]{Higher Chow cycles, cyclic cubic fourfolds and Lagrangian subvarieties}
\author[]{Shouhei Ma}
\author[]{Ken Sato}
\thanks{Supported by KAKENHI 21H00971} 
\address{Department~of~Mathematics, Tokyo~Institute~of~Technology, Tokyo 152-8551, Japan}
\email{ma.s.ae@m.titech.ac.jp}
\email{sato.k.da@m.titech.ac.jp}

\begin{abstract}
In this paper we initiate the study of higher Chow cycles on holomorphic symplectic manifolds. 
Our concrete central result is construction of explicit indecomposable $(2, 1)$- and $(4, 1)$-cycles on 
the Fano varieties of lines on cyclic cubic fourfolds. 
This is the first explicit example of such cycles on holomorphic symplectic manifolds. 
The proof of indecomposability is done by degeneration to cuspidal cubic fourfolds. 
Along the way, we develop a method of inducing $(p, 1)$-cycles on Hilbert squares of $K3$ surfaces. 
Finally, we study restriction of $(2, 1)$-cycles to Lagrangian subvarieties, 
and observe the phenomenon that the restricted cycles are always decomposable in the examples in our hand. 
\end{abstract}

\maketitle


\section{Introduction}\label{sec: intro}

Higher Chow groups of a smooth projective variety $X$ 
were introduced by Bloch \cite{Bl} 
as a cycle-theoretic incarnation of the motivic cohomology of $X$. 
They are expected to encode deep information, but remain mysterious in most cases. 
It is even not easy to construct explicit nontrivial cycles. 
In the past few decades, various examples and techniques of construction 
have been developed for surfaces, especially for $K3$ surfaces 
(see, e.g., \cite{Co}, \cite{CDKL}, \cite{MS}, \cite{Sa} and the references therein). 
It has been recognized that 
this line of investigation is related to and is benefited from 
the classical geometry of $K3$ surfaces. 

On the other hand, we know fewer examples in higher dimension. 
The situation in dimension $2$ suggests that  
a rich pool could be provided by higher dimensional analogues of $K3$ surfaces, 
\textit{holomorphic symplectic manifolds}. 
A more philosophical reason is that holomorphic $2$-forms govern the geometry of holomorphic symplectic manifolds, 
while they also play a key role in the study of higher Chow cycles of certain type. 

The purpose of this paper is to initiate the study of higher Chow cycles on 
holomorphic symplectic manifolds.  
In order to make a focus, 
we will mainly study $(p, 1)$-cycles on holomorphic symplectic fourfolds $X$, 
that is, elements of ${\CH}^{p}(X, 1)$. 
They are also called \textit{$K_1$-cycles}, 
as these higher Chow groups are constituents of the $K$-group $K_1(X)$ rationally (\cite{Bl}). 
Our more specific goals are threefold: 
\begin{enumerate}
\item Develop a method of inducing $K_1$-cycles on the Hilbert square $S^{[2]}$ of a $K3$ surface $S$ 
from $(2, 1)$-cycles on $S$ (\S \ref{sec: K3 induce}). 
\item Construct explicit examples of $(2, 1)$- and $(4, 1)$-cycles on the Fano varieties of lines on certain cubic fourfolds 
by combining (1) with a degeneration method (\S \ref{sec: deformation}, \S \ref{sec: indecomposable}). 
\item Study a new phenomenon on restriction of $(2, 1)$-cycles to Lagrangian subvarieties (\S \ref{sec: Lagrangian}). 
\end{enumerate}
The reader would notice that 
our choice of materials follows the genesis of 
holomorphic symplectic geometry, now with $K_1$-cycles. 

Since various examples of $(2, 1)$-cycles on $K3$ surfaces have been known, 
we obtain examples of $K_1$-cycles on their Hilbert squares by the method of (1). 
However, Hilbert squares of $K3$ surfaces form only a codimension $1$ locus in the moduli space: 
they deform to holomorphic symplectic fourfolds which are no longer Hilbert squares of $K3$ surfaces, 
and it is this fact that has enriched the study of holomorphic symplectic manifolds. 
Therefore it is desirable to construct cycles on holomorphic symplectic fourfolds 
by deforming induced cycles on $K3^{[2]}$. 
This is the purpose of (2). 

Our cubic fourfolds in (2) are cyclic covers of ${\proj}^4$ branched over cubic threefolds, known as 
\textit{cyclic cubic fourfolds}. 
The Fano varieties of such cubic fourfolds and their degeneration 
were studied recently Boissi\`ere, Sarti, Camere and Heckel \cite{BCS}, \cite{BHS}. 
We take the equation of the cubic fourfolds explicitly as 
\begin{equation}\label{eqn: cubic4 intro}
Y_{F,t} \: \: : \: \: x_4^3 - F(x_0, \cdots, x_3) + x_5(x_0x_3-x_1x_2) + t x_0x_5^2 = 0, 
\end{equation}
where $t\ne 0 \in \mathbb{A}^1$ and 
$F$ is a cubic form in $x_0, \cdots, x_3$. 
We normalize $F$ so that the genus $4$ curve cut out by $F$ on the quadric surface $x_0x_3=x_1x_2$ in ${\proj}^3$ 
is tangent to the lines $x_0=x_1=0$ and $x_0=x_2=0$. 
The cubic fourfold 
$Y_{F,t}$ is the cyclic cover of ${\proj}^4$ (with coordinates $x_0, \cdots, x_3, x_5$) 
branched over the cubic threefold  
\begin{equation*}
W_{F,t} \: \: : \: \: F(x_0, \cdots, x_3) - x_5(x_0x_3-x_1x_2) - t x_0x_5^2 = 0. 
\end{equation*}
It turns out that a general cubic threefold 
can be transformed to this form (Lemma \ref{lem: general cubic3}). 
Hence, by \cite{BCS}, the Fano varieties $F(Y_{F,t})$ of $Y_{F,t}$ 
form a locally complete family of holomorphic symplectic fourfolds \textit{with} a non-symplectic automorphism of order $3$ 
(given by $x_4\mapsto \exp(2\pi i/3)x_4$). 

A key point in our construction is the presence of two cuspidal plane cubics defined by 
\begin{equation*}
C_i \: \: : \: \: (x_0=x_i=x_5=0) \cap (x_4^3=F(x_0, \cdots, x_3)), \quad i=1, 2. 
\end{equation*}
It turns out that $Y_{F,t}$ contains the cone $\hat{C}_{i}$ over $C_i$ with vertex $[0, \cdots, 0, 1]$. 
This defines an embedding $C_i\hookrightarrow F(Y_{F,t})$. 
We can construct a $(4, 1)$-cycle $\xi_{4}$ on $F(Y_{F,t})$ from this curve configuration (\S \ref{ssec: (4,1)}). 
By considering the locus of lines meeting the cones $\hat{C}_{i}$, 
we also obtain a configuration of divisors on $F(Y_{F,t})$, 
from which we can construct a $(2, 1)$-cycle $\xi_{2}$ on $F(Y_{F,t})$  (\S \ref{ssec: (2,1)}). 
Then we prove the following (\S \ref{sec: indecomposable}). 

\begin{theorem}\label{thm: intro}
When $(F, t)$ is very general, 
the indecomposable parts of the $(2, 1)$-cycle $\xi_2$ and the $(4, 1)$-cycle $\xi_{4}$ 
on $F(Y_{F,t})$ are non-torsion. 
\end{theorem}

Here the \textit{indecomposable part} of a $(p, 1)$-cycle 
is its image in the cokernel of the cup product map 
${\rm CH}^{p-1}(X)\otimes_{\Z}{\C}^{\times} \to {\CH}^{p}(X, 1)$. 
Non-torsionness of the indecomposable part is a strong form of nontriviality of a $K_1$-cycle. 
A \textit{very general} point means a point in the complement of countably many divisors. 

Collino \cite{Co} is our predecessor who constructed infinitely many $(3, 1)$-cycles on general cubic fourfolds. 
Although not mentioned in \cite{Co}, 
this gives rise to $(2, 1)$-cycles on the Fano varieties by taking the Fano correspondence. 
Historically, this should be regarded as the first example of $(2, 1)$-cycles on holomorphic symplectic manifolds. 
Our $(2, 1)$-cycles are new example, somewhat more explicit. 
On the other hand, our $(4, 1)$-cycles would be the first example of such cycles. 
We could say that Theorem \ref{thm: intro} is the first \textit{explicit} example of $K_1$-cycles on holomorphic symplectic manifolds. 

There are both similarities and differences between the Collino construction and ours in several aspects. 
Firstly, both make use of cubic threefolds $W\subset Y$ cut out by a hyperplane, but in different ways. 
In our case, $W$ is the smooth branch divisor of $Y\to {\proj}^4$, 
while in Collino's case, $W$ is $4$-nodal, cut out by a special (but not explicit) hyperplane. 
Secondly, both use a cone $\hat{C}$ over a curve $C$, but the curves are different. 
In our case, $C$ is a union of cuspidal plane cubics, 
while in Collino's case, $C$ is an irreducible $3$-nodal elliptic curve. 
It might happen that our $(2, 1)$-cycles could be obtained as degeneration of one of Collino's cycles, 
but even so, our nontriviality would not follow from Collino's nontriviality; 
rather our cycles would connect Collino's cycles and $K3$ surfaces.   
Anyway, these intersections (some potential) 
would be reflection of richness of cubic fourfolds. 

We explain the idea of the proof of Theorem \ref{thm: intro}. 
The nontriviality of $\xi_{4}$ can be deduced from that of $\xi_{2}$ by using the Fano correspondence \textit{twice}. 
Thus we are reduced to the case of $\xi_2$, 
for which we use a degeneration method (again different from Collino's method). 
If we let $t\to 0$, the cubic threefold $W_{F,0}$ acquires a node, and accordingly, 
the cubic fourfold $Y_{F,0}$ acquires a cusp. 
A $K3$ surface $S$ appears here: 
as the one attached to the genus $4$ curve defined by $F$ by the method of Kond\=o \cite{Ko}. 
It was proved by Boissi\`ere-Heckel-Sarti \cite{BHS} that 
$S^{[2]}$ gives the symplectic resolution of the singular Fano variety $F(Y_{F,0})$. 
A key point is that our cycles $\xi_2$ on $F(Y_{F,t})$ specialize essentially to the cycle on $S^{[2]}$ 
induced from the $(2, 1)$-cycle on $S$ constructed in \cite{Sa}. 
Then the nontriviality proved in \cite{Sa} implies that of $\xi_2$ by a degeneration argument. 
Technically this argument is underpinned by the presence of non-symplectic automorphism, 
which compensates for the instability of the transcendental part of Abel-Jacobi map under deformation. 

The last topic studied in this paper is restriction of $(2, 1)$-cycles to Lagrangian subvarieties (\S \ref{sec: Lagrangian}). 
Given a $(2, 1)$-cycle $\xi$ on a holomorphic symplectic manifold $X$, 
it is natural to study $\xi$ by restricting it to the fibers of a Lagrangian fibration, 
or more generally, to Lagrangian subvarieties $W\subset X$. 
An elementary consideration (Lemma \ref{lem: Lagrangian J2}) raises the possibility that 
$\xi|_{W}$ might be always decomposable, i.e., 
contained in the image of the cup product 
${\rm Pic}(W)\otimes_{\Z}{\C}^{\times} \to {\CH}^{2}(W, 1)$. 
This is a new phenomenon, which trivially holds for $K3$ surfaces but 
seems far more nontrivial in dimension $>2$. 
In spite of the philosophy of holomorphic $2$-forms, 
we are still not confident enough for promoting this to a conjecture (cf.~Remark \ref{rmk: Beilinson Hodge}). 
Our goal is more modest: 
to confirm that this phenomenon holds in three examples in our hand. 
These are 
(i) the Lagrangian fibration associated to an elliptic $K3$ surface, with the induced cycles; 
(ii) the Lagrangian fibration associated to a $K3$ surface of degree $2$ (the Beauville-Mukai system), with the induced cycles; and 
(iii) the cycles $\xi_{2}$ on $F(Y)$ constructed in Theorem \ref{thm: intro}, 
with the Lagrangian surfaces being the Fano surfaces of the hyperplane sections of $Y$. 
In the first two cases the decomposability can be observed immediately, 
while in the last case we need more geometric consideration. 

The rest of this paper is organized as follows. 
\S \ref{sec: preliminary} is recollection of holomorphic symplectic manifolds and higher Chow cycles. 
In \S \ref{sec: K3 induce}, we construct induced cycles on Hilbert squares of $K3$ surfaces. 
\S \ref{sec: cuspidal cubic4} is recollection about $K3$ surfaces and singular cubic fourfolds associated to genus $4$ curves. 
In \S \ref{sec: deformation}, we construct our cycles $\xi_2$ and $\xi_4$. 
In \S \ref{sec: indecomposable}, we prove Theorem \ref{thm: intro}. 
In \S \ref{sec: Lagrangian}, we study restriction of $(2, 1)$-cycles to Lagrangian subvarieties.

We thank Yuji Odaka for instruction on simultaneous resolutions.

\section{Preliminaries}\label{sec: preliminary}

In this section we recall holomorphic symplectic manifolds and higher Chow cycles.

\subsection{Holomorphic symplectic manifolds}\label{ssec: IHS}

By a \textit{holomorphic symplectic manifold}, 
we mean a smooth projective variety $X$ which is simply-connected and 
whose $H^0(\Omega_{X}^{2})$ is spanned by a nowhere degenerate holomorphic $2$-form. 
Note that we assume projectivity in this paper. 
The second integral cohomology $H^2(X, \Z)$ is equipped with 
a certain integral symmetric bilinear form of signature $(3, \ast)$, 
known as the \textit{Beauville-Bogomolov-Fujiki form}.  
The N\'eron-Severi lattice ${\rm NS}(X)=H^{1,1}(X, \Z)$ has signature $(1, \ast)$, 
and is the orthogonal complement of $H^{2,0}(X)$ in $H^2(X, \Z)$. 

\subsection{Higher Chow cycles}\label{ssec: higher Chow}

Let $X$ be a holomorphic symplectic manifold. 
For $1\leq p \leq \dim(X)+1$, 
let ${\CH}^p(X, 1)$ be the higher Chow group of codimension $p$ and K-theoretic degree $1$ 
introduced by Bloch \cite{Bl}. 
An element of ${\CH}^p(X, 1)$ is called a \textit{$(p, 1)$-cycle}. 
%
Two presentations of ${\CH}^p(X, 1)$ are known: 
in terms of the Bloch complex (\cite{Bl}), 
and in terms of the Gersten complex (see, e.g., \cite{MullerStach}). 
We use the latter.  
Then a $(p, 1)$-cycle can be represented by a formal sum 
\begin{equation}\label{eqn: Gersten}
\xi = \sum_i (Z_i, \phi_i), 
\end{equation}
where $Z_i$ is an irreducible subvariety of $X$ of codimension $p-1$ 
and $\phi_i$ is a rational function on the normalization of $Z_i$, 
subject to the relation 
$\sum_{i} {\rm div}(\phi_i)=0$ as a codimension $p$ cycle on $X$ (the cocycle condition). 

We have the cup product map 
\begin{equation*}
{\CH}^{p-1}(X)\otimes_{\Z}{\C}^{\times} \to {\CH}^p(X, 1), \qquad Z\otimes \alpha \mapsto (Z, \alpha), 
\end{equation*}
where ${\C}^{\times}$ is identified with $H^0(X, \mathcal{O}_{X}^{\times}) = {\CH}^1(X, 1)$, 
the group of constant functions. 
We set 
\begin{equation*}
{\CH}^p(X, 1)_{\textrm{ind}} := {\CH}^p(X, 1)/ {\CH}^{p-1}(X)\otimes_{\Z}{\C}^{\times}. 
\end{equation*}
The image of a $(p, 1)$-cycle $\xi$ in ${\CH}^p(X, 1)_{\textrm{ind}}$ 
is called the \textit{indecomposable part} of $\xi$. 
If this is nonzero, $\xi$ is said to be \textit{indecomposable}.

\subsection{Regulators}\label{ssec: regulator}

A basic tool for studying higher Chow cycles is the regulator map. 
We define the \textit{Jacobian} of $X$ in degree $2p-2$ as 
\begin{equation*}
J^{2p-2}(X) := \frac{H^{2p-2}(X, \C)}{H^{2p-2}(X, \Z)+F^p}
\end{equation*}
where $(F^{\bullet})$ is the Hodge filtration. 
This is a generalized complex torus, i.e., the quotient of a ${\C}$-linear space by a lattice (not necessarily of full rank). 
%
%
The Jacobian $J^{2p-2}(X)$ contains 
\begin{equation*}
\frac{H^{p-1,p-1}(X, \Z)\otimes {\C}}{H^{p-1,p-1}(X, \Z)} 
= H^{p-1,p-1}(X, \Z)\otimes_{\Z} {\C}^{\times}
\end{equation*}
as a sub torus. 
We denote by 
\begin{equation*}
J^{2p-2}(X)_{{\rm tr}} := \frac{J^{2p-2}(X)}{H^{p-1,p-1}(X, \Z)\otimes_{\Z} {\C}^{\times}}
\end{equation*}
the quotient, and call it the \textit{transcendental Jacobian} of $X$. 

For simplicity, we assume that $H^{2p-1}(X, \Z)$ has no torsion. 
(This holds in the cases studied in this paper, see \cite{To}.) 
We have the Abel-Jacobi map 
\begin{equation*}
\nu : {\CH}^p(X, 1) \to J^{2p-2}(X) 
\end{equation*}
for $(p, 1)$-cycles, usually called the \textit{regulator map}. 
See \cite{KLM} for more details. 
We have the commutative diagram 
\begin{equation*}
\xymatrix{
{\CH}^{p-1}(X)\otimes_{\Z} {\C}^{\times} \ar[r] \ar[d] & {\CH}^{p}(X, 1) \ar[d]^{\nu} \\ 
H^{p-1,p-1}(X, \Z)\otimes_{\Z} {\C}^{\times} \ar[r] & J^{2p-2}(X), 
}
\end{equation*}
where the left vertical map is the cycle map for ${\CH}^{p-1}(X)$. 
Therefore the composition (the \textit{transcendental regulator map}) 
\begin{equation*}
\nu_{{\rm tr}} : {\CH}^p(X, 1) \to J^{2p-2}(X)_{{\rm tr}} 
\end{equation*}
descends to 
\begin{equation*}
{\CH}^p(X, 1)_{\textrm{ind}} \to J^{2p-2}(X)_{{\rm tr}}. 
\end{equation*}
This implies that, if $\nu_{{\rm tr}}(\xi)$ is nonzero or non-torsion for a $(p, 1)$-cycle $\xi$, 
so is the indecomposable part of $\xi$. 
This will be our way of detecting indecomposability. 


\subsection{Non-symplectic automorphisms}\label{ssec: auto}

One drawback of the transcendental regulator map $\nu_{{\rm tr}}$ 
is that it is not stable under deformation. 
In the case of $K3$ surfaces, it has been observed that 
non-symplectic automorphisms are useful for compensating for this matter (\cite{MS}, \cite{Sa}). 
We extend this formalism to holomorphic symplectic manifolds, specializing $p$ to $2$. 

Let $G$ be a cyclic group of prime order acting on a holomorphic symplectic manifold $X$. 
The $G$-action is called \textit{non-symplectic} if it acts nontrivially on the holomorphic $2$-forms. 
We denote by $H^2(X, \Z)^{G}$ the $G$-invariant part of $H^2(X, \Z)$. 
It is known that 
$H^2(X, \Z)^{G}$ is perpendicular to $H^{2,0}(X)$ and so $H^2(X, \Z)^{G}\subset {\rm NS}(X)$. 
Moreover, $H^2(X, \Z)^{G}$ has signature $(1, \ast)$. 

We define the Jacobian of $(X, G)$ in degree $2$ as 
\begin{equation*}
J^{2}(X, G) := \frac{J^{2}(X)}{H^2(X, \Z)^{G}\otimes_{\Z} {\C}^{\times}}. 
\end{equation*}
Since $H^2(X, \Z)^{G}\subset {\rm NS}(X)$, 
the projection $J^{2}(X)\to J^{2}(X)_{{\rm tr}}$ factors through $J^{2}(X, G)$: 
\begin{equation*}\label{eqn: G-Jacobian transcendental Jacobian}
J^{2}(X)\to J^{2}(X, G) \to J^{2}(X)_{{\rm tr}}. 
\end{equation*}
We have $J^{2}(X, G)=J^{2}(X)_{{\rm tr}}$ if and only if $H^2(X, \Z)^{G}={\rm NS}(X)$. 
Unlike $J^{2}(X)_{{\rm tr}}$, the $G$-Jacobian $J^{2}(X, G)$ has the advantage of being stable under deformation of $(X, G)$. 
We denote the composition of the regulator map $\nu$ and the projection $J^{2}(X)\to J^{2}(X, G)$ by 
\begin{equation*}\label{eqn: G-regulator}
\nu_{G} : {\CH}^2(X, 1)\to J^{2}(X, G), 
\end{equation*}
and call it the \textit{$G$-regulator map}. 
When $H^2(X, \Z)^{G}={\rm NS}(X)$, 
this coincides with $\nu_{{\rm tr}}$. 


\section{Induced cycles from $K3$ surfaces}\label{sec: K3 induce}

In this section, we develop a method of constructing $(p, 1)$-cycles on 
the Hilbert square of a $K3$ surface $S$ from $(2, 1)$-cycles on $S$. 
Some part is valid for Hilbert schemes of larger length, 
but we restrict ourselves to the case of Hilbert square for the sake of simplicity of exposition. 
Since various examples of $(2, 1)$-cycles on $K3$ surfaces 
have been known (see, e.g., \cite{Co}, \cite{CDKL}, \cite{MS}, \cite{Sa} and the references therein), 
we can obtain indecomposable $(p, 1)$-cycles on their Hilbert squares in this way.

\subsection{$(2, 1)$-cycles}\label{ssec: (2,1) K3}

Let $S$ be a projective $K3$ surface. 
We denote by $S^{[2]}$ the Hilbert scheme parametrizing $0$-dimensional subschemes of $S$ of length $2$. 
We have the commutative diagram 
\begin{equation*}
\xymatrix{
{\rm Bl}_{\Delta}S^{2} \ar[r]^{f} \ar[d]_{g} & S^{[2]} \ar[d] \\ 
S^2 \ar[r] & S^{2}/\frak{S}_2
}
\end{equation*}
where $\Delta\subset S^2$ is the diagonal, 
the vertical maps are the blow-up along the diagonal, 
and the horizontal maps are the quotient map by the symmetric group $\frak{S}_{2}$. 
We denote by $\pi \colon S^2 \to S$ the projection to the first factor. 

It is known (\cite{Be}) that 
\begin{equation*}\label{eqn: H correspondence S2}
f_{\ast} \circ g^{\ast} \circ \pi^{\ast} \: : \: 
H^2(S, \Z) \to H^2(S^{[2]}, \Z) 
\end{equation*}
is an embedding of quadratic lattices, 
and gives the orthogonal decomposition 
\begin{equation}\label{eqn: H2S2}
H^2(S^{[2]}, \Z) = H^2(S, \Z) \oplus \Z \delta, 
\end{equation}
where $2\delta =[f(g^{-1}(\Delta))]$ is the class of the exceptional divisor of 
the Hilbert-Chow morphism $S^{[2]}\to S^2/\frak{S}_2$. 
This shows that 
\begin{equation}\label{eqn: J2S2}
J^2(S^{[2]}) \simeq J^2(S) \: \times \: {\C}^{\times} [\delta]. 
\end{equation}

\begin{lemma}\label{prop: induced (2,1)}
The following diagram commutes: 
\begin{equation}\label{eqn: regulator S2}
\xymatrix@C=44pt{
{\CH}^{2}(S, 1) \ar[r]^{f_{\ast} \circ g^{\ast} \circ \pi^{\ast}} \ar[d]_{\nu} & {\CH}^{2}(S^{[2]}, 1) \ar[d]^{\nu} \\ 
J^2(S) \ar@{^{(}->}[r] & J^{2}(S^{[2]}), 
}
\end{equation}
where $J^2(S) \hookrightarrow J^{2}(S^{[2]})$ is induced by \eqref{eqn: J2S2}. 
\end{lemma}

\begin{proof}
This follows from the compatibility of regulator maps with 
pullback and proper pushforward (see \cite{We} \S 3). 
\end{proof}

For a $(2, 1)$-cycle $\xi$ on $S$, we denote 
\begin{equation*}
\xi^{[2]} := f_{\ast} \circ g^{\ast} \circ \pi^{\ast} (\xi ) \quad \in \: {\CH}^2(S^{[2]}, 1). 
\end{equation*}

\begin{corollary}\label{cor: induced (2,1)}
If $\nu_{{\rm tr}}(\xi)$ is nonzero or non-torsion, so is $\nu_{{\rm tr}}(\xi^{[2]})$. 
\end{corollary}

\begin{proof}
By composing \eqref{eqn: regulator S2} with the projections to the transcendental Jacobians, 
we obtain the commutative diagram 
\begin{equation*}
\xymatrix@C=44pt{
{\CH}^{2}(S, 1) \ar[r]^{f_{\ast} \circ g^{\ast} \circ \pi^{\ast}} \ar[d]_{\nu_{{\rm tr}}} & {\CH}^{2}(S^{[2]}, 1) \ar[d]^{\nu_{{\rm tr}}} \\ 
J^2(S)_{{\rm tr}} \ar[r]^{\simeq} & J^{2}(S^{[2]})_{{\rm tr}}. 
}
\end{equation*} 
Here $J^2(S)_{{\rm tr}} \to J^{2}(S^{[2]})_{{\rm tr}}$ is an isomorphism by \eqref{eqn: H2S2}. 
Thus $\nu_{{\rm tr}}(\xi^{[2]})$ is the image of $\nu_{{\rm tr}}(\xi)$ by this isomorphism. 
This proves our assertion. 
\end{proof}

The following explicit description of $\xi^{[2]}$ will be useful. 
Suppose that $\xi$ is presented as 
$\xi=\sum_{i}(C_i, \phi_i)$ as in \eqref{eqn: Gersten} 
where $C_i$ is an irreducible curve on $S$ and $\phi_i$ is a rational function on (the normalization of) $C_i$. 
We define a divisor of $S^{[2]}$ as 
\begin{equation*}
C_i+S := f(\widetilde{C_i\times S}) \; \;  \subset \: S^{[2]}, 
\end{equation*}
where $\widetilde{C_i\times S}$ is the strict transform of 
$C_i\times S \subset S^2$ in ${\rm Bl}_{\Delta}S^2$. 
Note that $\widetilde{C_i\times S}$ is also the pullback of $C_i\times S$, 
as $C_i\times S$ intersects properly with $\Delta$. 
Since the projection $\widetilde{C_i\times S}\to C_i+S$ is birational, 
we can define a rational function on $C_i+S$ by pulling back $\phi_i$ by 
\begin{equation*}
C_i+S \sim \widetilde{C_i\times S} \stackrel{g}{\to} C_i\times S \stackrel{\pi}{\to} C_i. 
\end{equation*}
We denote it by $\phi_{i}^{[2]}$. 

\begin{lemma}\label{lem: explicit description}
The $(2, 1)$-cycle $\xi^{[2]}$ is represented by $\sum_{i}(C_i+S, \phi_{i}^{[2]})$. 
\end{lemma}

\begin{proof}
Since the blow-up map 
$g\colon {\rm Bl}_{\Delta}S^2 \to S^2$ 
is not flat, 
we need to recall explicit description of the pullback by $g$ on $(2, 1)$-cycles. 
By the construction of blow-up, 
$g$ factorizes as 
${\rm Bl}_{\Delta}S^2 \stackrel{i}{\hookrightarrow} {\proj}_{S^2}\stackrel{p}{\to} S^2$ 
where ${\proj}_{S^2}$ is a ${\proj}^2$-bundle over $S^2$ and $i$ is a regular closed immersion. 
Then the pullback by $g$ is the following process: 
(1) take the flat pullback by $p$; 
(2) replace it by a rationally equivalent cycle which intersects properly with ${\rm Bl}_{\Delta}S^2$; and  
(3) take the naive restriction to ${\rm Bl}_{\Delta}S^2 \subset {\proj}_{S^2}$. 
Note that the process (2) is in general assured by Bloch's moving lemma \cite{Bl2}. 

In our case $\xi=\sum_{i}(C_i, \phi_i)$, 
we have 
\begin{equation*}
p^{\ast} \circ \pi^{\ast} (\xi) \: = \: 
\sum_{i} (p^{\ast}(C_i\times S), \: \phi_i \circ \pi \circ p) 
\end{equation*}
by the flat pullback. 
Since $C_i\times S$ and ${\rm div}(\phi_i \circ \pi )$ intersect properly with $\Delta\subset S^2$, 
we see that $p^{\ast}(C_i\times S)$ and ${\rm div}(\phi_i \circ \pi \circ p)$ 
intersect properly with ${\rm Bl}_{\Delta}S^2 \subset {\proj}_{S^2}$. 
This means that $p^{\ast} \circ \pi^{\ast} (\xi)$ already satisfies the proper intersection condition. 
Therefore we have 
\begin{equation*}
g^{\ast} \circ \pi^{\ast} (\xi) =
\sum_{i}(\widetilde{C_i \times S}, \: \phi_i \circ \pi \circ g). 
\end{equation*}

Finally, by the proper pushforward, we have 
\begin{equation*}
f_{\ast} \circ g^{\ast} \circ \pi^{\ast} (\xi) =
\sum_{i}(C_{i}+S, \: N(\phi_i \circ \pi \circ g)) 
\end{equation*}
where $N$ is the norm map from the function field of $\widetilde{C_i\times S}$ to that of $C_i + S$. 
Since $\widetilde{C_i\times S}\to C_i + S$ is birational, 
we have $N(\phi_i \circ \pi \circ g)= \phi_{i}^{[2]}$. 
\end{proof}

Later we need the $G$-regulator version of 
Lemma \ref{prop: induced (2,1)} and Corollary \ref{cor: induced (2,1)}. 
Suppose that we have a non-symplectic action of a cyclic group $G$ of prime order on $S$. 
The induced $G$-action on $S^{[2]}$ is also non-symplectic. 

\begin{lemma}\label{prop: G-regulator (2,1)}
We have the following commutative diagram: 
\begin{equation*}
\xymatrix@C=44pt{
{\CH}^{2}(S, 1) \ar[r]^{f_{\ast} \circ g^{\ast} \circ \pi^{\ast}} \ar[d]_{\nu_{G}} & {\CH}^{2}(S^{[2]}, 1) \ar[d]^{\nu_{G}} \\ 
J^2(S, G) \ar[r]^{\simeq} & J^{2}(S^{[2]}, G), 
}
\end{equation*}
where the lower isomorphism is induced by \eqref{eqn: J2S2}.  
In particular, 
if $\nu_{G}(\xi)$ is nonzero or non-torsion, so is $\nu_{G}(\xi^{[2]})$. 
\end{lemma}

\begin{proof}
This is obtained by composing \eqref{eqn: regulator S2} with the projections to the respective $G$-Jacobians. 
Note that the relevant maps $\pi$, $g$, $f$ are all $G$-equivariant. 
The induced map $J^2(S, G) \to J^{2}(S^{[2]}, G)$ is an isomorphism because 
$H^2(S^{[2]}, \Z)^{G} = H^2(S, \Z)^{G} \oplus \Z \delta$. 
\end{proof}

\subsection{$(4, 1)$-cycles}\label{ssec: (4,1) K3}

Let $S$ and $\xi=\sum_{i}(C_i, \phi_i)$ be as in \S \ref{ssec: (2,1) K3}. 
Next we construct $(4, 1)$-cycles on $S^{[2]}$. 
We take a point $p\in S$ which is not contained in the support of ${\rm div}(\phi_i)$ for any $i$. 
Let $\widetilde{C_{i}\times p}$ be the strict transform of the curve 
$C_i\times p \subset S^2$ in ${\rm Bl}_{\Delta}S^2$,  
and $C_{i}+p$ be the image of $\widetilde{C_{i}\times p}$ in $S^{[2]}$. 
Since $C_{i}+p$ is birational to $C_i$, 
we can regard $\phi_i$ as a rational function on (the normalization of) $C_{i}+p$. 
We use the same notation $\phi_i$ for it. 
Then we put 
\begin{equation*}
\xi_{p} := \sum_{i} (C_{i}+p, \: \phi_i). 
\end{equation*}
By the assumption on $p$, $\xi_{p}$ satisfies the cocycle condition 
and so defines a $(4, 1)$-cycle on $S^{[2]}$. 
When $p\not\in C_i$ for any $i$, 
we see from Lemma \ref{lem: explicit description} that 
$\xi_{p}$ is the intersection of the $(2, 1)$-cycle $\xi^{[2]}$ with the $2$-cycle $p+S$ on $S^{[2]}$. 
The cycle $\xi_{p}$ is defined even when $p\in C_i$; 
we only require $p\not\in {\rm div}(\phi_i)$.

\begin{proposition}\label{prop: trans regulator (4,1)}
If $\nu_{{\rm tr}}(\xi)$ is nonzero or non-torsion, so is $\nu_{{\rm tr}}(\xi_{p})$. 
\end{proposition}

\begin{proof}
Let $\xi_{p}'$ be the $(4, 1)$-cycle 
$\sum_{i}(\widetilde{C_{i}\times p}, \phi_i)$ on ${\rm Bl}_{\Delta}S^2$. 
By construction we have 
\begin{equation}\label{eqn: xip'}
(\pi \circ g)_{\ast} \xi_{p}' = \xi, \quad f_{\ast}\xi_{p}'=\xi_{p}. 
\end{equation}
We consider the diagram 
\begin{equation*}
\xymatrix@C=40pt{
{\CH}^{2}(S, 1) \ar[d]_{\nu} & 
{\CH}^{4}({\rm Bl}_{\Delta}S^2, 1) \ar[l]_{(\pi \circ g)_{\ast}}  \ar[r]^{f_{\ast}} \ar[d]^{\nu} & {\CH}^{4}(S^{[2]}, 1) \ar[d]^{\nu} \\ 
J^2(S)  & J^{6}({\rm Bl}_{\Delta}S^2) \ar[l]^{(\pi \circ g)_{\ast}} \ar[r]_{f_{\ast}} & J^6(S^{[2]}) 
}
\end{equation*}
which commutes by \cite{We} \S 3. 
By \eqref{eqn: xip'}, we have 
\begin{equation*}
(\pi \circ g)_{\ast} \nu( \xi_{p}') = \nu(\xi), \quad f_{\ast}\nu(\xi_{p}') = \nu(\xi_{p}). 
\end{equation*}
Taking the projection to the transcendental Jacobians, we obtain 
\begin{equation}\label{eqn: nutrxip}
(\pi \circ g)_{\ast} \nu_{{\rm tr}}( \xi_{p}') = \nu_{{\rm tr}}(\xi), \quad 
f_{\ast}\nu_{{\rm tr}}(\xi_{p}') = \nu_{{\rm tr}}(\xi_{p}). 
\end{equation}

Now we shall show that $\nu_{{\rm tr}}(\xi)\ne 0$ implies $\nu_{{\rm tr}}(\xi_{p})\ne 0$. 
(The case of non-torsion is similar.) 
Suppose that $\nu_{{\rm tr}}(\xi_{p})= 0$. 
Then $f_{\ast}\nu_{{\rm tr}}(\xi_{p}')=0$ by \eqref{eqn: nutrxip}. 
Since $f^{\ast}\circ f_{\ast}=1+\iota_{\ast}$ where 
$\iota$ is the switch involution of ${\rm Bl}_{\Delta}S^2$, 
we obtain 
$\nu_{{\rm tr}}(\xi_{p}')+\iota_{\ast}\nu_{{\rm tr}}(\xi_{p}')=0$. 
Sending this equality by $(\pi\circ g)_{\ast}$, we obtain 
\begin{eqnarray*}
0 & = & 
(\pi\circ g)_{\ast}\nu_{{\rm tr}}(\xi_{p}')+(\pi\circ g \circ \iota)_{\ast}\nu_{{\rm tr}}(\xi_{p}') \\ 
& = & \nu_{{\rm tr}}(\xi) + \nu_{{\rm tr}}((\pi\circ g \circ \iota)_{\ast}\xi_{p}') \\ 
& = & \nu_{{\rm tr}}(\xi). 
\end{eqnarray*}
Here the second equality follows from \eqref{eqn: nutrxip}, 
and the last equality holds because $\iota_{\ast}\xi_{p}'$ is contracted to the point $p$ by $\pi \circ g$. 
\end{proof}

The $G$-regulator version holds similarly.

\subsection{$(3, 1)$-cycles}\label{ssec: (3,1) K3}

Let $S$ and $\xi=\sum_{i}(C_i, \phi_i)$ be as before. 
Next we construct $(3, 1)$-cycles. 
Let $C$ be an irreducible curve on $S$ which is disjoint from ${\rm div}(\phi_i)$ for any $i$. 
In particular, $C\ne C_i$ for any $i$. 
We can define the surface $C_i+C$ of $S^{[2]}$ as before, 
which is naturally birational to $C_i\times C$. 
Then 
\begin{equation*}
\xi_{C} := \sum_i(C_i+C, \: \phi_{i}) 
\end{equation*}
is a $(3, 1)$-cycle on $S^{[2]}$, 
where $\phi_{i}$ stands for the induced rational function 
\begin{equation*}
C_i+C \sim C_i\times C \to C_i \stackrel{\phi_i}{\dashrightarrow} {\proj}^1. 
\end{equation*}
By the computation of intersection product 
\begin{equation*}
(C_i+S. \, C+S) = (C_i+C) + \sum_{p\in C_i\cap C}(S+p),  
\end{equation*}
we see from Lemma \ref{lem: explicit description} that  
\begin{equation}\label{eqn: (3,1) intersection}
\xi_{C} = (\xi^{[2]}. \, C+S) - \sum_{i}\sum_{p\in C_i\cap C}(S+p, \: \phi_i(p)), 
\end{equation}
where $\phi_i(p)$ in the last term means a constant function on $S+p$. 
Thus, modulo decomposable cycles, $\xi_{C}$ is essentially the intersection product of 
$\xi^{[2]}$ with the divisor $C+S$. 

\begin{proposition}
If $\nu_{{\rm tr}}(\xi)$ is non-torsion and $C$ is ample, 
then $\nu_{{\rm tr}}(\xi_C)$ is non-torsion. 
\end{proposition}

\begin{proof}
We write $D$ for the divisor $C+S$ of $S^{[2]}$. 
By \eqref{eqn: (3,1) intersection}, 
we have $\nu_{{\rm tr}}(\xi_{C})=\nu_{{\rm tr}}(\xi^{[2]}. \, D)$. 
By the compatibility of regulator map and intersection product (\cite{We}), 
we see that $\nu(\xi^{[2]}. \, D)$ is the image of $\nu(\xi^{[2]})$ under the map 
$J^2(S^{[2]})\to J^4(S^{[2]})$ induced from the cup product 
$H^2(S^{[2]}, \Z)\to H^4(S^{[2]}, \Z)$ with $[D]$. 
Since $\nu_{{\rm tr}}(\xi^{[2]})$ is non-torsion by Corollary \ref{cor: induced (2,1)}, 
it suffices to show that the induced map 
$J^2(S^{[2]})_{{\rm tr}}\to J^4(S^{[2]})_{{\rm tr}}$ has finite kernel. 

Since $C$ is ample, $D$ is the pullback of an ample divisor on $S^2/\frak{S}_{2}$. 
Since $S^{[2]}\to S^2/\frak{S}_2$ is a semi-small resolution, 
we see from \cite{dCM} that 
the Hard Lefschetz theorem holds for $D$, that is, 
the cup product with $[D]^2$ 
\begin{equation*}
H^2(S^{[2]}, \Q) \longrightarrow  H^6(S^{[2]}, \Q) 
\end{equation*}
is an isomorphism of $\Q$-Hodge structures. 
Thus the composition 
\begin{equation*}
J^2(S^{[2]})_{{\rm tr}} \stackrel{(\cdot . D)}{\longrightarrow}
J^4(S^{[2]})_{{\rm tr}} \stackrel{(\cdot . D)}{\longrightarrow}
J^6(S^{[2]})_{{\rm tr}} 
\end{equation*}
has finite kernel. 
This proves our assertion. 
\end{proof}

\section{Eisenstein $K3$ surfaces and cuspidal cyclic cubic fourfolds}\label{sec: cuspidal cubic4}

This section is preliminaries for the next \S \ref{sec: deformation}. 
In \S \ref{ssec: EK3}, we recall Kond\=o's construction of Eisenstein $K3$ surfaces attached to genus $4$ curves \cite{Ko}. 
In \S \ref{ssec: ccc4}, we recall cuspidal cubic fourfolds attached to those $K3$ surfaces following 
Boissi\`ere-Heckel-Sarti \cite{BHS}. 
In \S \ref{ssec: (2,1) EK3}, we recall the $(2, 1)$-cycles on those $K3$ surfaces constructed in \cite{Sa}. 

Throughout this paper, the Fano variety of lines on a cubic fourfold $Y\subset {\proj}^5$ 
will be denoted by $F(Y)$. 
We use the notation $[x_0, \cdots, x_n]$ for the homogeneous coordinates on ${\proj}^n$.

\subsection{Eisenstein $K3$ surfaces}\label{ssec: EK3}

Let $Q\subset {\proj}^3$ be the smooth quadric surface defined by $x_0x_3=x_1x_2$. 
We take a general cubic form $F(x_0, \cdots, x_3)$. 
It cuts out a canonical genus $4$ curve $B$ on $Q$. 
By taking the triple cyclic cover of $Q$ branched over $B$, we obtain an Eisenstein $K3$ surface $(S, G)$, 
i.e., a $K3$ surface with a non-symplectic automorphism of order $3$ (\cite{Ko}). 
Explicitly, $S$ is defined as the $(2, 3)$ complete intersection 
\begin{equation}\label{eqn: EK3}
S  \: \: : \: \:  (x_4^3=F(x_0, \cdots, x_3)) \: \cap \: (x_0x_3=x_1x_2) 
\end{equation}
in ${\proj}^4$. 
The non-symplectic automorphism is given by 
$x_4\mapsto \exp(2\pi i/3) x_4$. 
Note that the quadric defined by $x_0x_3=x_1x_2$ in ${\proj}^4$ is a cone over $Q$.

\subsection{Cuspidal cyclic cubic fourfolds}\label{ssec: ccc4}

Let $S\subset {\proj}^4$ be as in \eqref{eqn: EK3}. 
We define a cubic fourfold by 
\begin{equation}\label{eqn: ccc4}
Y \: \: : \: \: x_4^3 - F(x_0, \cdots, x_3) + x_5(x_0x_3-x_1x_2) \: = \: 0. 
\end{equation}
Such a cubic fourfold is called a \textit{cuspidal cyclic cubic fourfold} and 
studied extensively by Boissi\`ere-Heckel-Sarti \cite{BHS}. 
We refer to \cite{BHS} for the following properties of $Y$, 
which are extension of those for nodal cubic fourfolds \cite{Ha}. 

When $Y$ is general, it has an ordinary cusp point at $p_0=[0, \cdots, 0, 1]$ as the only singularity. 
For every point $p\in S\subset {\proj}^4$, the line $\overline{p_{0}p}$ is contained in $Y$. 
Conversely, every line on $Y$ passing through $p_0$ is of this form. 
Thus we have the embedding 
\begin{equation*}
S\hookrightarrow F(Y), \quad p\mapsto \overline{p_0p}, 
\end{equation*} 
whose image is the locus of lines through $p_0$. 
It turns out that this is the singular locus of $F(Y)$, 
and $F(Y)$ has transversal $A_2$-singularities along $S$.  

We assume that $S$ does not contain a line. 
Then we have the morphism 
\begin{equation}\label{eqn: S2 F(Y)}
S^{[2]}\to F(Y), \qquad p_1+p_2\mapsto l_{p_{1}p_{2}}, 
\end{equation}
where $l_{p_{1}p_{2}}$ is the residual line of 
$\overline{p_0p_1}+\overline{p_0p_2}$; 
recall that this is the line defined by $Y\cap P=\overline{p_0p_1}+\overline{p_0p_2}+l_{p_{1}p_{2}}$, 
where $P$ is the plane spanned by $p_0, p_1, p_2$. 
The morphism \eqref{eqn: S2 F(Y)} is the blow-up of $F(Y)$ along the singular locus $S$. 
This gives the symplectic resolution of $F(Y)$. 

The exceptional divisor $E=E_1+E_2$ of \eqref{eqn: S2 F(Y)} is 
a union of two ${\proj}^1$-bundles over $S$ intersecting at the respective distinguished section ($\simeq S$). 
Geometrically this is described as follows (cf.~\cite{Hu} \S 6.1.5). 
Firstly, the choice of a component, say $E_{\alpha}$, 
corresponds to the choice of a ruling on $Q$. 
Given a point $p$ of $S$, let $P$ be the plane spanned by $p_0$ and 
the line on $Q$ that passes through the image of $p$ and belongs to the chosen ruling. 
Then the fiber ${\proj}^1$ of $E_{\alpha}\to S$ over $p$ corresponds to 
the pencil of lines $l$ on $P$ passing through $p$, 
with the distinguished line $\overline{p_0p}$. 
The point of $S^{[2]}$ corresponding to $l$ is the remaining two intersection points of 
$l$ and the plane cubic $S\cap P$. 
Note that $S\cap P$ is a fiber of the elliptic fibration on $S$ 
obtained as the pullback of $Q\to {\proj}^1$.

Finally, we note that the cubic fourfold $Y$ is the cyclic cover of ${\proj}^4$ 
(with coordinates $x_0, \cdots, x_3, x_5$) 
branched over the nodal cubic threefold 
\begin{equation}\label{eqn: nodal cubic3}
F(x_0, \cdots, x_3) - x_5(x_0x_3-x_1x_2) = 0. 
\end{equation}
The construction in \S \ref{ssec: EK3} and \S \ref{ssec: ccc4} can be roughly summarized as 
\begin{equation*}
\xymatrix{
\textrm{genus 4 curves} \ar@{<.>}[r] \ar@{<.>}[d] & \textrm{Eisenstein K3 surfaces} \ar@{<.>}[d] \\ 
\textrm{nodal cubic threefolds} \ar@{<.>}[r] & \textrm{cuspidal cyclic cubic fourfolds}
}
\end{equation*}
Here the horizontal constructions are "taking cyclic cover of the ambient variety", 
and the vertical constructions are "taking the locus of lines through the singular point".

\subsection{$(2, 1)$-cycles}\label{ssec: (2,1) EK3}

Let $(S, G)$ be an Eisenstein $K3$ surface as in \eqref{eqn: EK3}. 
We recall the $(2, 1)$-cycle on $S$ constructed in \cite{Sa}. 
We need to choose two intersecting lines on $Q$ which are tangent to the branch curve $B$. 
It will be convenient to normalize these two lines to 
$R_1=(x_0=x_1=0)$ and $R_2=(x_0=x_2=0)$ in ${\proj}^3$. 
(Under the standard isomorphism $Q\simeq {\proj}^1\times {\proj}^1$, 
they correspond to $\{0 \}\times {\proj}^1$ and ${\proj}^1\times \{0 \}$.) 
Thus we assume that 
$B$ is tangent to both $R_1$ and $R_2$ with multiplicity $2$. 
A general canonical genus $4$ curve can be transformed to this form 
(with finitely many choices of the line pairs). 
We also assume that $B$ does not pass through $R_1\cap R_2 = [0, 0, 0, 1]$. 

Let $\pi \colon S\to Q$ be the projection. 
Then $C_i:= \pi^{-1}(R_i)$ is a cuspidal plane cubic curve on ${\proj}^4$. 
Explicitly, it is defined by 
\begin{equation}\label{eqn: Ci}
C_i \: \: : \: \: (x_0=x_i=0)  \: \cap \:  (x_4^3=F(x_0, \cdots, x_3)). 
\end{equation}
We label the intersection points $C_1 \cap C_2 = \pi^{-1}(R_1\cap R_2)$ by writing 
\begin{equation*}
C_1 \cap C_2 = \{ p_+, p_-, p_{\infty} \}. 
\end{equation*}
By our assumption $R_1\cap R_2\not\in B$, the points $p_+, p_-, p_{\infty}$ are distinct. 
Note also that $p_+, p_-, p_{\infty}$ are collinear: 
they lie on the line $x_0=x_1=x_2$. 
See Figure 1 for the configuration of $C_1$ and $C_2$.

\begin{figure}[h]\label{figure: C1C2}
\includegraphics[height=60mm, width=50mm]{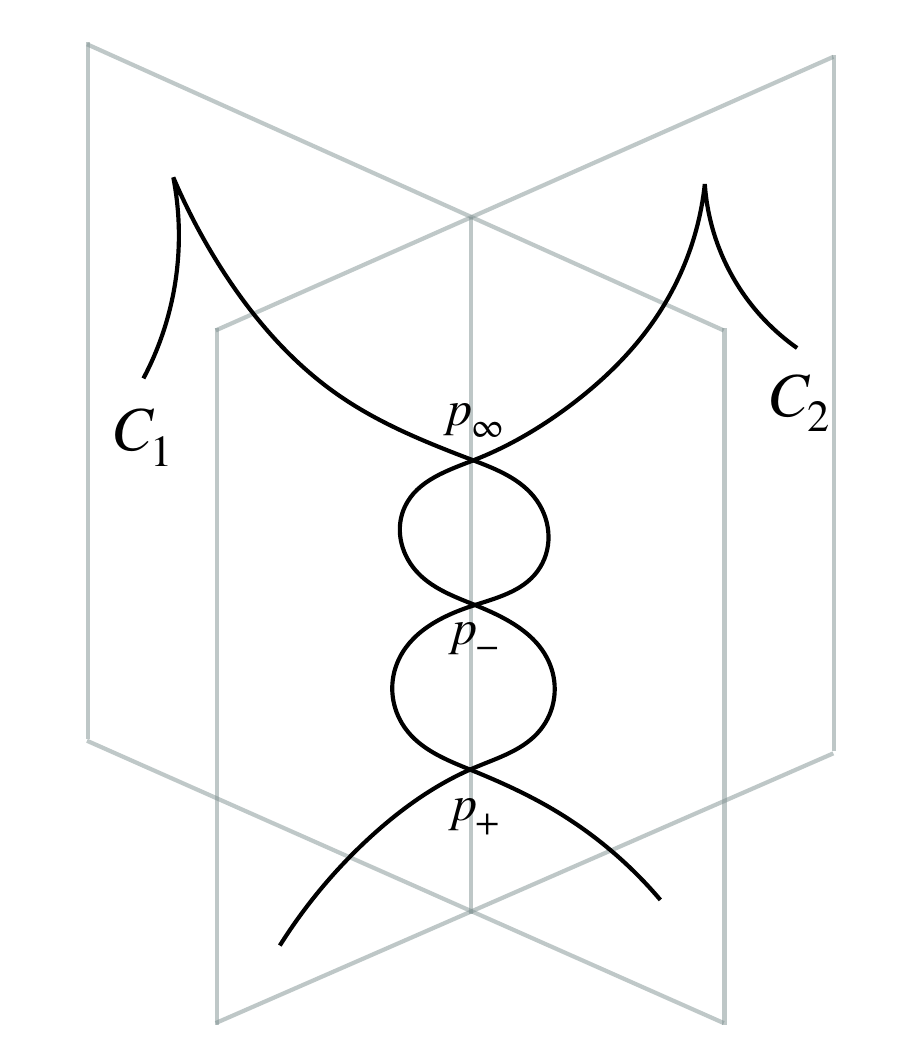}
\caption{}
\end{figure}

Since the normalization of $C_i$ is ${\proj}^1$, 
we can take a rational function $\phi_i$ on $C_i$ such that 
${\rm div}(\phi_1) = p_+ -p_-$ and ${\rm div}(\phi_2) = p_- -p_+$. 
Geometrically, $\phi_{i}$ is given by the projection from the cusp of $C_i$ 
on the plane $x_0=x_i=0$. 
It will be useful to normalize $\phi_{i}$ by setting $\phi_i(p_{\infty})=1$. 
Then 
\begin{equation*}
\xi = (C_1, \phi_1) +  (C_2, \phi_2) 
\end{equation*}
is a $(2, 1)$-cycle on $S$. 
It is proved in \cite{Sa} that the indecomposable part of $\xi$ is non-torsion for very general $F$. 
If we look back the proof, it actually says the following. 

\begin{theorem}[\cite{Sa}]\label{thm: EK3}
For very general $F$, 
the $G$-regulator of the $(2, 1)$-cycle $\xi$ on $S$ is non-torsion. 
\end{theorem} 

Recall from \S \ref{sec: K3 induce} that $\xi$ gives rise to a $(2, 1)$-cycle $\xi^{[2]}$ on $S^{[2]}$.  

\begin{corollary}\label{cor: EK3}
When $F$ is very general, 
the $G$-regulator of $\xi^{[2]}$ is non-torsion. 
In particular, for very general $F$, 
the indecomposable part of $\xi^{[2]}$ is non-torsion. 
\end{corollary}

\begin{proof}
The first assertion follows from 
Theorem \ref{thm: EK3} and Lemma \ref{prop: G-regulator (2,1)}. 
The second assertion follows from the first assertion and the fact that 
${\rm NS}(S^{[2]})=H^{2}(S^{[2]}, \Z)^{G}$ 
as far as ${\rm NS}(S)=H^{2}(S, \Z)^{G}$, i.e., for very general $(S, G)$. 
\end{proof}


\section{Deformation to smooth cubic fourfolds}\label{sec: deformation}

In this section, we construct $(2, 1)$- and $(4, 1)$-cycles on a family of Fano varieties of smooth cubic fourfolds with $\Z/3$-action. 
The proof of indecomposability will be given in the next \S \ref{sec: indecomposable}. 
It will turn out (in \S \ref{sec: indecomposable}) that 
our $(2, 1)$-cycles specialize to the induced cycles in Corollary \ref{cor: EK3} 
if we let the cubic fourfolds degenerate.

\subsection{A family of cubic fourfolds}\label{ssec: cub4 family}

We begin by setting up our family of cubic fourfolds. 
Let $U_0$ be the space of cubic forms $F(x_0, \cdots, x_3)$ in variables $x_0, \cdots, x_3$ 
which satisfy the assumptions in \S \ref{ssec: (2,1) EK3}, namely 
the curve cut out by $F$ on $Q\subset {\proj}^3$ is smooth, does not pass through $[0, 0, 0, 1]$, 
and is tangent to the lines $R_i=(x_0=x_i=0)$, $i=1, 2$, with multiplicity $2$. 
Let $U = U_0\times \mathbb{A}^1$. 
We use $t$ for expressing coordinates on $\mathbb{A}^1$. 
Then we define a family $Y\to U$ of cubic fourfolds by the equation 
\begin{equation}\label{eqn: cub4}
x_4^3 - F(x_0, \cdots, x_3) + x_5(x_0x_3-x_1x_2) + t x_0 x_5^2 \: = \: 0 
\end{equation}
in $U\times {\proj}^5$. 
We denote by $Y_{F,t}$ the fiber of $Y\to U$ over $(F, t)\in U$. 
Note that every $Y_{F,t}$ contains the point $p_0=[0, \cdots, 0, 1]$. 
When $(F, t)$ is general, $Y_{F,t}$ is smooth. 
The restriction of $Y\to U$ over $U_0$ is the family \eqref{eqn: ccc4} of cuspidal cyclic cubic fourfolds 
considered in \S \ref{ssec: ccc4}. 
In what follows, we freely shrink $U_0\subset U$ to suitable Zariski open sets (intersecting with $U_0$) 
without changing the notation $U_0\subset U$. 

The group $G=\Z/3$ acts on $Y\to U$ by $x_4\mapsto \exp(2\pi i/3) x_4$. 
Let $H_0\subset {\proj}^5$ be the hyperplane defined by $x_4=0$. 
This is ${\proj}^4$ with coordinates $x_0, \cdots, x_3, x_5$. 
The cubic fourfold $Y_{F,t}$ is the cyclic cover of $H_0$ 
branched over the cubic threefold 
\begin{equation*} 
W_{F,t} \: \: : \: \: F(x_0, \cdots, x_3) - x_5(x_0x_3-x_1x_2) - t x_0 x_5^2  \: = \: 0 
\end{equation*}
with Galois group $G$. 
When $(F, t)$ is general, $W_{F,t}$ is smooth. 
When $t=0$, $W_{F,0}$ is the nodal cubic threefold in \eqref{eqn: nodal cubic3}. 

\begin{lemma}\label{lem: general cubic3}
The rational map $(F, t)\mapsto W_{F,t}$ 
from $U$ to the moduli space of smooth cubic threefolds is dominant. 
\end{lemma}

\begin{proof}
Let $\mathcal{M}_{{\rm cub}3}$ be the moduli space of smooth cubic threefolds and 
$\mathcal{M}_{{\rm cub}3}'$ be the partial compactification of $\mathcal{M}_{{\rm cub}3}$ 
obtained by adding one-nodal cubics (cf.~\cite{Al}, \cite{Yo}). 
The morphism $U-U_0 \to \mathcal{M}_{{\rm cub}3}$ extends to 
$U\to \mathcal{M}_{{\rm cub}3}'$. 
As explained in the last paragraph of \S \ref{ssec: ccc4}, 
the boundary morphism 
$U_0\to \mathcal{M}_{{\rm cub}3}'-\mathcal{M}_{{\rm cub}3}$ 
is dominant. 
Since $\mathcal{M}_{{\rm cub}3}'-\mathcal{M}_{{\rm cub}3}$ is of codimension $1$ in $\mathcal{M}_{{\rm cub}3}'$, 
we find that $U\to \mathcal{M}_{{\rm cub}3}'$ is dominant. 
\end{proof}

\begin{corollary}\label{cor: G-inv=NS}
We have 
$H^2(F(Y_{F,t}), \Z)^{G} = \Z \mathcal{O}(1)$ 
where $\mathcal{O}(1)$ is the Pl\"ucker line bundle on $F(Y_{F,t})\subset \mathbb{G}(1, 5)$. 
When $(F, t)$ is very general, this coincides with ${\rm NS}(F(Y_{F,t}))$. 
\end{corollary}

\begin{proof}
This is proved in \cite{BCS} for smooth cubic fourfolds obtained as 
cyclic covers of ${\proj}^4$ branched over smooth cubic threefolds. 
\end{proof}

\subsection{Two cones}\label{ssec: cycle family}

A key point in our construction is the presence of two cones in $Y_{F,t}$. 
We begin by explaining this. 
In what follows, we fix $(F, t)$ and write $Y=Y_{F,t}$. 
We define cuspidal plane cubic curves in ${\proj}^5$ by  
\begin{equation*}
C_i \: \; : \: \; (x_0=x_i=x_5=0) \: \cap \: (x_4^3=F(x_0, \cdots, x_3)), \quad i=1, 2. 
\end{equation*}
These are the same curves as \eqref{eqn: Ci}, and do not depend on $t$.  
Clearly the cubic fourfold $Y$ contains $C_1$, $C_2$. 
Let $\hat{C}_{i}$ be the cone over $C_i$ with vertex $p_0=[0, \cdots, 0, 1]$. 

\begin{lemma}\label{lem: cone in Yt}
We have $\hat{C}_{1}, \hat{C}_{2}\subset Y$. 
\end{lemma} 

\begin{proof}
What has to be checked is $\overline{p_0p}\subset Y$ for any $p\in C_1$. 
(The case of $C_2$ is similar.) 
By the equation of $C_1$, a point $p$ of $C_1$ has coordinates 
$[0, 0, a_2, a_3, a_4, 0]$ where $a_2, a_3, a_4$ satisfy $a_4^3=F(0, 0, a_2, a_3)$. 
Points on the line $\overline{p_0p}$ (different from $p_0$) have coordinates 
$[0, 0, a_2, a_3, a_4, \lambda]$ where $\lambda\in \mathbb{A}^1$.  
Since the $x_0$-coordinate and the $x_1$-coordinate are $0$, 
this point satisfies the equation \eqref{eqn: cub4} of $Y$. 
\end{proof}

By Lemma \ref{lem: cone in Yt}, we obtain the embedding 
\begin{equation}\label{eqn: Ci in F(Y) naive}
C_i \hookrightarrow F(Y), \quad p\mapsto \overline{p_0p}. 
\end{equation}
We denote by $\Gamma_{i}\subset F(Y)$ its image. 

\begin{lemma}\label{lem: C(t)}
Let $t\ne 0$. 
Then $\Gamma_{1} + \Gamma_{2}$ coincides with the locus of lines passing through $p_0$. 
\end{lemma}

\begin{proof}
Let $l$ be a line on $Y$ passing through $p_0$, 
and $p=[a_0, \cdots, a_4, 0]$ be the intersection point of $l$ with the hyperplane $x_5=0$. 
Points on $l$ different from $p_0$ have coordinates 
$[a_0, \cdots, a_4, \lambda]$ with $\lambda\in \mathbb{A}^1$. 
Since $l\subset Y$, we have 
\begin{equation*}
(a_4^3-F(a_0, \cdots, a_3)) + (a_0a_3-a_1a_2)\lambda + ta_0\lambda^2 \equiv 0 
\end{equation*}
for any $\lambda$ by the equation \eqref{eqn: cub4} of $Y$. 
Looking this as a polynomial of $\lambda$ and noticing $t\ne 0$, we see that 
\begin{equation*}
a_0=0, \quad a_1a_2=0, \quad a_4^3=F(a_0, \cdots, a_3). 
\end{equation*}
This means $p\in C_1$ or $p\in C_2$ depending on whether $a_1=0$ or $a_2=0$. 
Therefore $l=\overline{p_0p}$ belongs to $\Gamma_1+\Gamma_2$. 
\end{proof}
 
For two points $p_1, p_2\in C_i$, 
the lines $\overline{p_0p_1}$, $\overline{p_0p_2}$ are contained in $Y$ and intersect at $p_0$. 
This determines the residual line, which we denote by $l_{p_1p_2}$. 
Recall that this is defined by 
$Y\cap P = \overline{p_0p_1} + \overline{p_0p_2} + l_{p_1p_2}$ 
where $P$ is the plane spanned by $p_0, p_1, p_2$. 
Note that $l_{p_1p_2}$ is defined even when $p_1=p_2$: 
in that case, $P$ is the tangent plane of the cone $\hat{C}_{i}$ along $\overline{p_0p_1}=\overline{p_0p_2}$.  

\begin{lemma}\label{lem: lp1p2}
For $p_1, p_2\in C_i$, we have $l_{p_1p_2}=\overline{p_0p_3}$ for the third intersection point $p_3$ 
of $\overline{p_1p_2}$ with $C_i$. 
In particular, $l_{p_1p_2}$ belongs to $\Gamma_{i}$. 
\end{lemma}

\begin{proof}
We consider only the case $i=1$. 
Since $C_1$ is a cubic curve on the plane $x_0=x_1=x_5=0$, 
the line $\overline{p_1p_2}$ intersects with $C_1$ at a third point: this is $p_3$. 
Since $p_1, p_2, p_3$ are collinear, 
the line $\overline{p_0p_3}$ is contained in the plane $P$ spanned by $p_0, p_1, p_2$. 
On the other hand, since $p_3\in C_1$, we have $\overline{p_0p_3}\subset Y$ by Lemma \ref{lem: cone in Yt}. 
Thus $\overline{p_0p_3}\subset P\cap Y$. 
By the definition of $l_{p_1p_2}$, 
this means $l_{p_1p_2}=\overline{p_0p_3}$. 
\end{proof}

As in \S \ref{ssec: (2,1) EK3}, we label the intersection points of $C_1$ and $C_2$ by writing 
\begin{equation*}\label{eqn: C1C2 label}
C_1\cap C_2 = \{ p_+, p_-, p_{\infty}\}. 
\end{equation*} 
Then we have 
\begin{equation*}
\Gamma_{1} \cap \Gamma_{2} = \{ \overline{p_{0}p_{+}}, \: \overline{p_{0}p_{-}}, \: \overline{p_{0}p_{\infty}} \}. 
\end{equation*}

\begin{remark}
Strictly speaking, to define this labeling globally requires to take an etale cover of $U$. 
This is presumed in what follows (without changing the notation $U$). 
\end{remark}


\subsection{$(4, 1)$-cycles}\label{ssec: (4,1)}

Now we construct a $(4, 1)$-cycle on $F(Y)=F(Y_{F,t})$ with $t\ne 0$. 
Let $\phi_i$ be the rational function on $C_i$ taken in \S \ref{ssec: (2,1) EK3}. 
Recall that this satisfies ${\rm div}(\phi_1)=p_+-p_-$ and ${\rm div}(\phi_2)=p_{-}-p_{+}$, 
and is normalized by setting the value at $p_{\infty}$ to be $1$. 
We regard $\phi_i$ as a rational function on $\Gamma_{i}$ via the isomorphism $C_i\to \Gamma_{i}$, 
and denote it again by $\phi_i$. 
Then 
\begin{equation*}
\xi_4 := (\Gamma_1, \phi_1) + (\Gamma_2, \phi_2) 
\end{equation*}
satisfies the cocycle condition and hence defines a $(4, 1)$-cycle on $F(Y)$. 

\begin{theorem}\label{thm: (4,1)}
When $(F, t)$ is very general, 
the $(4, 1)$-cycle $\xi_4$ on $F(Y_{F,t})$ has non-torsion transcendental regulator.   
In particular, its indecomposable part is non-torsion. 
\end{theorem}

The proof will be given in \S \ref{ssec: nontrivial (4,1)}.

\subsection{$(2, 1)$-cycles}\label{ssec: (2,1)}

Next we construct a $(2, 1)$-cycle. 
Again let $t\ne 0$.  
We denote by $D_i\subset F(Y)$ the locus of lines intersecting with $\hat{C}_i$. 
We write $D_{i}^{\circ}\subset D_i$ for the Zariski open set parametrizing  
lines which intersect with $\hat{C}_{i}-p_0$ at only one point. 
We have the natural morphism 
\begin{equation}\label{eqn: Di hatCi}
D_i^{\circ} \to \hat{C}_{i}-p_0, \quad l\mapsto l\cap \hat{C}_{i}. 
\end{equation}

\begin{lemma}\label{lem: Dicirc}
We have 
$D_{i} - D_{i}^{\circ} = \Gamma_{1}+\Gamma_{2}$. 
\end{lemma}

\begin{proof} 
A line $l$ in $D_{i} - D_{i}^{\circ}$ either passes through $p_0$ or 
intersects with $\hat{C}_{i}-p_0$ at more than one point. 
In the former case, $l$ belongs to $\Gamma_{1}+\Gamma_{2}$ by Lemma \ref{lem: C(t)}. 
In the latter case, we take two points $p_1\ne p_2$ in the intersection of $l$ with $\hat{C}_i-p_0$. 
If $\overline{p_0p_1}=\overline{p_0p_2}$, then $l=\overline{p_1p_2}$ passes through $p_0$. 
If $\overline{p_0p_1}\ne \overline{p_0p_2}$, 
then $l$ is the residual line of $\overline{p_0p_1}+\overline{p_0p_2}$ 
and hence belongs to $\Gamma_{i}$ by Lemma \ref{lem: C(t)}. 
\end{proof}

The relation of $D_i$ and $\hat{C}_i$ can be explained better 
by considering the universal line 
\begin{equation}\label{eqn: univ line}
F(Y) \stackrel{\pi_{F}}{\longleftarrow} \mathbb{L} \stackrel{\pi_{Y}}{\longrightarrow} Y. 
\end{equation}
Then we have 
\begin{equation}\label{eqn: DiCi}
D_{i} = \pi_{F}( \pi_{Y}^{-1}(\hat{C}_i)). 
\end{equation}
Lemma \ref{lem: Dicirc} shows that 
the projection $\pi_{F}\colon \pi_{Y}^{-1}(\hat{C}_i) \to D_{i}$ is isomorphic over the complement of $\Gamma_{i}$, 
and contracts the lines over $\Gamma_{i}$. 

By \cite{CS} \S 2, 
the projection $\pi_{Y}\colon \mathbb{L}\to Y$ is flat over 
the complement of (at most) finitely many points called \textit{Eckardt points}. 
They are points $p$ such that the intersection of $Y$ 
with the (projective) tangent hyperplane $T_pY$ at $p$ is a cone with vertex $p$; 
the $\pi_{Y}$-fiber over such a point $p$ is a surface, the base of the cone.  
It is known that generic cubic fourfolds have no Eckardt point (\cite{HRS}). 
In our case, however, the cyclic cubic fourfolds are \textit{not} generic. 
Nevertheless we can say the following, 
which is sufficient for our purpose. 

\begin{lemma}\label{lem: Eckardt}
Suppose that $Y$ contains no plane. 
Then $Y$ has no Eckardt point on $\hat{C}_1+\hat{C}_2$. 
\end{lemma}

\begin{proof}
Let $p\in \hat{C}_{i}$. 
By Lemma \ref{lem: C(t)}, we may assume $p\ne p_0$. 
We take a general hyperplane section $C_i'=\hat{C}_i\cap H$ of $\hat{C}_i$ through $p$. 
This is a cuspidal plane cubic isomorphic to $C_i$.  
We first consider the case $p$ is not on the singular line of $\hat{C}_i$. 
By our assumption, $Y$ does not contain the tangent plane of $\hat{C}_i$ at $p$. 
This shows that $T_pC_i'\not\subset Y$ for a generic choice of $H$. 
Then $Y\cap T_pC_i'=C_i'\cap T_pC_i'=2p+q$ for some $q\ne p\in C_i'$. 
On the other hand, $T_pC_i'$ is contained in $T_pY$. 
Thus $q$ is contained in $Y\cap T_pY$ but $\overline{pq}$ is not. 
Hence $Y\cap T_pY$ is not a cone. 

Next suppose that $p$ is on the singular line of $\hat{C}_i$. 
Then $p$ is the cusp of $C_i'$. 
For general $q\ne p \in C_i'$ the line $\overline{pq}$ is not contained in $Y$, 
for otherwise $Y$ would contain the linear hull of $C_i'$. 
On the other hand, $\overline{pq}$ is contained in $T_pY$ because it intersects with $Y$ at $p$ with multiplicity $2$. 
Thus $q\in Y\cap T_pY$ but $\overline{pq}\not\subset Y\cap T_pY$. 
Hence $Y\cap T_pY$ is not a cone again. 
\end{proof}

By Lemma \ref{lem: Eckardt}, 
$Y=Y_{F,t}$ has no Eckardt point on $\hat{C}_1+\hat{C}_2$ for general $(F, t)$.  
Moreover, a general fiber of $\pi_{Y}^{-1}(\hat{C}_i) \to \hat{C}_i$ is irreducible for general $(F, t)$. 
In what follows, we assume this genericity condition. 
Then $\pi_{Y}^{-1}(\hat{C}_i)$ is flat over $\hat{C}_i$, 
so in particular it is irreducible of dimension $3$. 
It follows from \eqref{eqn: DiCi} that $D_i$ is an irreducible divisor. 

We consider the composition of \eqref{eqn: Di hatCi} and the projection $\hat{C}_i \dashrightarrow C_i$: 
\begin{equation}\label{eqn: Di(t) to Ci}
D_i \stackrel{\eqref{eqn: Di hatCi}}{\dashrightarrow} \hat{C}_{i}\dashrightarrow C_i.  
\end{equation}
The indeterminacy locus is $\Gamma_{j}$ where $j\ne i$. 
We define a rational function on $D_i$, denoted by $\tilde{\phi_{i}}$, 
as the pullback of $\phi_i$ by \eqref{eqn: Di(t) to Ci}. 
If we denote by $W_{\pm}\subset D_i^{\circ}$ the locus of lines intersecting with $\overline{p_0p_{\pm}}$, 
we have 
\begin{equation}\label{eqn: div (2,1)}
{\rm div}(\tilde{\phi}_{1})|_{D_{1}^{\circ}} = W_{+} - W_{-}, \quad 
{\rm div}(\tilde{\phi}_{2})|_{D_{2}^{\circ}} = W_{-} - W_{+}. 
\end{equation}
Now we can define a $(2, 1)$-cycle on $F(Y)$ by 
\begin{equation*}
\xi_2 = (D_1, \tilde{\phi}_{1}) + (D_2, \tilde{\phi}_{2}). 
\end{equation*}
By \eqref{eqn: div (2,1)}, this satisfies the cocycle condition. 

\begin{theorem}\label{thm: (2,1)}
When $(F, t)$ is very general, 
the $(2, 1)$-cycle $\xi_2$ on $F(Y_{F,t})$ has non-torsion $G$-regulator. 
In particular, for very general $(F, t)$, 
the transcendental regulator of $\xi_{2}$ is non-torsion   
and so $\xi_{2}$ has non-torsion indecomposable part. 
\end{theorem}

The proof will be given in \S \ref{ssec: nontrivial (2,1)}.

\section{Proof of indecomposability}\label{sec: indecomposable}

In this section we prove Theorem \ref{thm: (2,1)} and Theorem \ref{thm: (4,1)}. 
We first prove Theorem \ref{thm: (2,1)} in \S \ref{ssec: nontrivial (2,1)} by a degeneration method. 
Then, in \S \ref{ssec: nontrivial (4,1)}, we deduce Theorem \ref{thm: (4,1)} from Theorem \ref{thm: (2,1)} 
by using the Fano correspondence.

\subsection{Proof of Theorem \ref{thm: (2,1)}}\label{ssec: nontrivial (2,1)}

In this subsection we prove Theorem \ref{thm: (2,1)}. 
The second assertion follows from the first by Corollary \ref{cor: G-inv=NS}, 
so it suffices to prove the first assertion. 
Since the Jacobians $J^2(X, G)$ are stable under deformation and the normal functions are holomorphic, 
we see that the locus where the $G$-regulators of our cycles are torsion 
is a union of countably many analytic sets. 
Thus we are reduced to verifying non-emptyness of its complement, i.e., 
existence of $(F, t)$ such that $\xi_2$ has non-torsion $G$-regulator. 
We do this by a degeneration method. 

In what follows, we fix a very general $F\in U_0$ according to Corollary \ref{cor: EK3}. 
Let $\Delta\subset \mathbb{A}^1$ be a small disc around $t=0$. 
We restrict the family $Y\to U$ over $\{ F \} \times \Delta \subset U$ and rewrite 
$Y|_{\{ F \} \times \Delta} \to \Delta$ as $Y\to \Delta$. 
We denote by $Y_{t}=Y_{F,t}$ the fiber over $t\in \Delta$. 
In order to specify $t$, 
we write $D_i=D_i(t)$, $\tilde{\phi_i} = \tilde{\phi_i}(t)$ and $\xi_2 = \xi_2(t)$. 
 
Let $(S, G)$ be the Eisenstein $K3$ surface associated to $Y_0$ in the sense of \S \ref{ssec: ccc4}. 
By Namikawa's theory \cite{Na}, after taking a base change if necessary, 
we may take a simultaneous resolution 
$X\to F(Y/\Delta)$ of the relative Fano variety. 
Our cycle family $(\xi_{2}(t))_{t\ne 0}$ is defined over $\Delta^{\ast}=\Delta-\{ 0 \}$. 
On the other hand, recall from \S \ref{ssec: (2,1) EK3} that 
we have the induced $(2, 1)$-cycle 
\begin{equation*}
\xi^{[2]} = (C_1+S, \: \phi_{1}^{[2]}) + (C_2+S, \: \phi_{2}^{[2]}) 
\end{equation*}
on the central fiber $X_0=S^{[2]}$. 
(See Lemma \ref{lem: explicit description} for the notation.) 
Recall that the exceptional divisor $E=E_1+E_2$ of $S^{[2]}\to F(Y_0)$ 
is a union of two ${\proj}^1$-bundles over $S$. 

\begin{lemma}\label{lem: limit (2,1)}
We have $\lim_{t\to 0} \xi_2(t) = \xi^{[2]} + \xi'$ on $X_0$ where $\xi'$ is a $(2, 1)$-cycle supported on $E$. 
\end{lemma}

\begin{proof}
Note that $D_{i}(t)^{\circ}$ is defined even when $t=0$. 
Then $D_i(0)^{\circ}\subset F(Y_0)$ is the limit of $D_i(t)^{\circ}$ 
and is disjoint from the singular locus of $F(Y_0)$ 
(the locus of lines through $p_0$). 

\begin{claim}\label{claim}
The closure of $D_i(0)^{\circ}$ in $X_0=S^{[2]}$ is the divisor $C_i+S$, 
and the limit of $\tilde{\phi_{i}}(t)$ on $D_i(0)^{\circ}$ coincides with the induced function $\phi_{i}^{[2]}$ on $C_i+S$. 
\end{claim}

\begin{proof} 
Since $C_i+S$ is irreducible, it suffices to verify $D_i(0)^{\circ}\subset C_i+S$ for the first assertion. 
Let $[l]\in D_i(0)^{\circ}$. 
Let $\hat{S}$ be the cone over $S$ with vertex $p_0$. 
The fact that $S^{[2]}\to F(Y_0)$ is isomorphic outside the singular locus of $F(Y_0)$ 
means that $l$ intersects with $\hat{S}$  at exactly two points. 
One is $l\cap \hat{C}_i$: we denote by $p$ the corresponding point on $C_i$, 
i.e., $\overline{p_0p}$ passes through $l\cap \hat{C}_i$. 
We denote by $q$ the point on $S$ corresponding to the other intersection point. 
Then $l$ is the residual line of $\overline{p_0p}+\overline{p_0q}$. 
This means $[l]=p+q\in C_i+S$ in $S^{[2]}$. 

This argument also shows that the induced rational function 
\begin{equation*}
\phi_{i}^{[2]} \: : \: C_i + S \dashrightarrow C_i \stackrel{\phi_{i}}{\dashrightarrow} {\proj}^1 
\end{equation*}
coincides with 
\begin{equation*}
D_i(0) \stackrel{\eqref{eqn: Di(t) to Ci}}{\dashrightarrow} C_i \stackrel{\phi_{i}}{\dashrightarrow} {\proj}^1,  
\end{equation*}
and so is the limit of $\tilde{\phi_i}(t)$ on $C_1+S$. 
\end{proof}

We go back to the proof of Lemma \ref{lem: limit (2,1)}. 
By Claim \ref{claim}, the limit of $D_i(t)$ in $X_0$ 
is the sum of $C_i+S$ and a divisor with support $E$. 
Let $E_{\alpha}$ be a component of $E$. 
It remains to verify that $\tilde{\phi_{i}}(t)$ has a limit on $E_{\alpha}$. 
(Strictly speaking, it is more convenient to pass to the Bloch complex 
when discussing specialization of $(2, 1)$-cycles, 
but we omit it for simplicity.) 

Let $D_i$ be the closure of $\bigcup_{t\ne0}D_i(t)$ in $X$. 
This is an irreducible divisor of $X$ whose central fiber is $(C_i+S)\cup E$. 
Then $(\tilde{\phi_{i}}(t))_{t\ne 0}$ defines a rational function on $D_{i}$, 
which we denote by $\tilde{\phi_{i}}$. 
We first consider the case $D_i$ is normal at a general point of $E_{\alpha}$. 
What has to be shown is that $\tilde{\phi_{i}}|_{E_{\alpha}}$ is well-defined, i.e., 
$E_{\alpha} \not\subset {\rm div}(\tilde{\phi_{i}})$.  
(This corresponds to the proper intersection condition in the Bloch complex.) 
If $E_{\alpha}\subset {\rm div}(\tilde{\phi_{i}})$ to the contrary, 
then 
\begin{equation}\label{eqn: divisor inclusion}
{\rm div}(\tilde{\phi_{i}}|_{C_i+S}) = {\rm div}(\tilde{\phi_{i}}) \cap (C_i+S) \supset E_{\alpha} \cap (C_i+S). 
\end{equation}
On the other hand, a geometric consideration shows that 
\begin{equation*}
E_{\alpha} \cap (C_i+S) = 
\begin{cases}
C_i + C_i & \alpha=i, \\ 
C_i\times_{{\proj}^1}S & \alpha\ne i. 
\end{cases}
\end{equation*}
Here $S\to {\proj}^1$ is the elliptic fibration corresponding to $\alpha$, 
of which $C_i$ is a trisection. 
Note that $C_i+S$ is non-normal at $C_i+C_i$ with two branches. 
Clearly the induced function $\phi_{i}^{[2]}$ on $C_i+S$ 
does not contain these surfaces in its divisor. 
This contradicts with \eqref{eqn: divisor inclusion}. 
Hence $E_{\alpha} \not\subset {\rm div}(\tilde{\phi_{i}})$. 

When $D_i$ is non-normal at $E_{\alpha}$, 
we take the normalization $\hat{D_i}\to D_i\subset X$ of $D_i$. 
Let $\hat{E}_{\alpha}$ be the inverse image of $E_{\alpha}$ in $\hat{D}_{i}$ and 
$\hat{E}_{\alpha}=\sum_{j}\hat{E}_{\alpha,j}$ be its irreducible decomposition. 
By the above argument, the limit of $\tilde{\phi_{i}}$ on each $\hat{E}_{\alpha,j}$ exists. 
Taking its norm under $\hat{E}_{\alpha,j} \to E_{\alpha}$, 
we obtain a rational function on $E_{\alpha}$, say $\tilde{\phi}_{i,j}$. 
Then $\prod_{j}\tilde{\phi}_{i,j}$ gives the desired function on $E_{\alpha}$, namely 
$\sum_{j}(E_{\alpha}, \tilde{\phi}_{i,j}) = (E_{\alpha}, \prod_{j}\tilde{\phi}_{i,j})$ 
gives the limit chain of $(D_{i}(t), \tilde{\phi_{i}}(t))$ on $E_{\alpha}$. 
\end{proof}


We resume the proof of Theorem \ref{thm: (2,1)}.  
Since $E_1\cap E_2\simeq S$ is an irreducible divisor in both $E_1$ and $E_2$, 
a consideration of the cocycle condition shows that a $(2, 1)$-cycle supported on 
$E=E_1+E_2$ must be decomposable. 
Hence, writing $\xi(0) = \lim_{t\to 0} \xi_2(t)$, we find from Lemma \ref{lem: limit (2,1)} that 
\begin{equation*}
\nu_{G}(\xi_2(0)) = \nu_{{\rm tr}}(\xi_2(0)) = \nu_{{\rm tr}}(\xi^{[2]}) = \nu_{G}(\xi^{[2]}), 
\end{equation*}
and this is non-torsion by Corollary \ref{cor: EK3}. 

Now we can do a degeneration argument like \cite{MS}, \cite{Sa} as follows. 
We denote by $\pi \colon X\to \Delta$ the projection. 
This is a smooth family of holomorphic symplectic manifolds 
with $G$-action over $\Delta^{\ast}$ and over $t=0$. 
(The $G$-action ``jumps' at $t=0$.) 
Then $H^{2}(X_t, \Z)^{G}$ for $t\ne0$ form a rank $1$ sub local system of $R^{2}\pi_{\ast}\Z |_{\Delta^{\ast}}$. 
Explicitly, by Corollary \ref{cor: G-inv=NS}, the stalks are spanned by $\mathcal{O}(1)$. 
By the triviality of $R^{2}\pi_{\ast}\Z$, this local system over $\Delta^{\ast}$ 
extends to a sub local system $\mathcal{H}$ of $R^{2}\pi_{\ast}\Z$ over $\Delta$. 
The stalk of $\mathcal{H}$ over $t=0$ is spanned by the pullback of $\mathcal{O}(1)$ by $X_0\to F(Y_0)$. 
In particular, it is contained in $H^2(X_0, \Z)^{G}$. 
Therefore, if $\mathcal{J}\to \Delta$ is the Jacobian fibration associated to $R^{2}\pi_{\ast}\Z/\mathcal{H}$, 
we have $\mathcal{J}_{t}=J^2(F(Y_t), G)$ for $t\ne 0$, 
while we have a natural surjection $\mathcal{J}_0\twoheadrightarrow J^2(S^{[2]}, G)$ for $t=0$. 

The regulator of $\xi_{2}(t)$ defines a holomorphic section 
$(\nu(t))_t$ of $\mathcal{J}$ over $\Delta$. 
The central value $\nu(0)\in \mathcal{J}_0$ is sent to 
$\nu_{G}(\xi_2(0)) = \nu_{G}(\xi^{[2]})$ by $\mathcal{J}_0\twoheadrightarrow J^2(S^{[2]}, G)$. 
Since $\nu_{G}(\xi^{[2]})$ is non-torsion as noted above, 
we find that $\nu(0)$ is non-torsion. 
By the holomorphicity of $\nu$, this shows that 
$\nu(t)=\nu_G(\xi_2(t))$ is non-torsion for very general $t\ne 0$. 
This proves Theorem \ref{thm: (2,1)}.

\subsection{Proof of Theorem \ref{thm: (4,1)}}\label{ssec: nontrivial (4,1)}

In this subsection we deduce Theorem \ref{thm: (4,1)} from Theorem \ref{thm: (2,1)}. 
Let $\mathbb{L}$ be the universal line as in \eqref{eqn: univ line}.  
We use the correspondence by $\mathbb{L}$ in both directions: 
\begin{equation*}
\Phi^{F\to Y} := \pi_{Y\ast}\circ \pi_{F}^{\ast} \; : \; 
{\CH}^4(F(Y), 1) \to {\CH}^3(Y, 1), 
\end{equation*} 
\begin{equation*}
\Phi^{Y\to F} := \pi_{F \ast}\circ \pi_{Y}^{\ast} \; : \; 
{\CH}^3(Y, 1) \to {\CH}^2(F(Y), 1). 
\end{equation*} 

\begin{lemma}\label{lem: correspondence}
We have 
$\Phi^{Y\to F}(\Phi^{F\to Y}(\xi_4))=\xi_2$. 
\end{lemma}

\begin{proof}
We have 
\begin{equation*}
\Phi^{F\to Y}(\xi_4) = (\hat{C}_1, \hat{\phi}_1) + (\hat{C}_2, \hat{\phi}_2), 
\end{equation*}
where $\hat{\phi}_i$ is the composition of the projection $\hat{C}_i\dashrightarrow C_i$ with $\phi_i$. 
On the other hand, we see from \eqref{eqn: DiCi} that 
$\pi_{F\ast}\pi_{Y}^{\ast} \hat{C}_{i} = D_i$. 
Hence, by the definition of $\phi_i(t)$, we have 
\begin{equation*}
\xi_{2} = \Phi^{Y\to F}((\hat{C}_1, \hat{\phi}_1) + (\hat{C}_2, \hat{\phi}_2)). 
\end{equation*}
This proves our assertion. 
\end{proof}

By taking the regulator, Lemma \ref{lem: correspondence} tells us that 
$\nu(\xi_2)$ is the image of $\nu(\xi_4)$ under the map 
\begin{equation*}
J^6(F(Y)) \to J^4(Y) \to J^2(F(Y)) 
\end{equation*}
induced by the correspondence by $\mathbb{L}$ (twice). 
In particular, $\nu_{{\rm tr}}(\xi_2)$ is the image of $\nu_{{\rm tr}}(\xi_4)$ by a homomorphism 
$J^6(F(Y))_{{\rm tr}}\to J^2(F(Y))_{{\rm tr}}$. 
In this way, 
non-torsionness of $\nu_{{\rm tr}}(\xi_2)$ (proved in Theorem \ref{thm: (2,1)}) 
implies that of $\nu_{{\rm tr}}(\xi_4)$. 
This proves Theorem \ref{thm: (4,1)}.

\section{$(2, 1)$-cycles on Lagrangian subvarieties}\label{sec: Lagrangian}

Let $X$ be a holomorphic symplectic manifold. 
A smooth subvariety $W\subset X$ with $\dim X=2\dim W$ is called \textit{Lagrangian} 
if the restriction of the holomorphic $2$-form on $X$ to $W$ is identically zero. 
Since Lagrangian fibrations are higher dimensional analogues of elliptic fibrations, 
one natural approach for studying $(2, 1)$-cycles on $X$ is to restrict them to 
the fibers of a Lagrangian fibration, or more generally, to Lagrangian subvarieties. 

The following observation implies that we can no longer use regulators to show that 
the restricted cycles are indecomposable. 

\begin{lemma}\label{lem: Lagrangian J2}
Let $W\subset X$ be a Lagrangian subvariety. 
The restriction map $J^2(X)_{{\rm tr}}\to J^2(W)_{{\rm tr}}$ is zero. 
\end{lemma}

\begin{proof}
Recall that $H^2(X, \Z)/{\rm NS}(X)$ is a Hodge structure of weight two. 
Since $h^{2,0}(X)=1$, it contains no proper sub $\Z$-Hodge structure. 
The restriction map 
\begin{equation*}
H^2(X, \Z)/{\rm NS}(X) \to H^2(W, \Z)/{\rm NS}(W) 
\end{equation*}
is a morphism of Hodge structures. 
By the Lagrangian property, this annihilates $H^{2,0}(X)$. 
Hence it is the zero map.  
\end{proof}

If one believes that holomorphic $2$-forms play a dominant role in the study of $(2, 1)$-cycles, 
this observation even raises the possibility that the restricted cycles might be always decomposable 
(cf.~Remark \ref{rmk: Beilinson Hodge}). 
We are still not confident enough for formulating this as a conjecture. 
Instead, in the following, we observe that this property holds in the three examples in our hand.

\begin{example}
Let $\pi\colon S\to {\proj}^1$ be an elliptic $K3$ surface and 
$\xi$ be a $(2, 1)$-cycle on $S$. 
Then $\pi$ induces the Lagrangian fibration 
\begin{equation*}\label{eqn: Lag ell}
S^{[2]} \to {\proj}^2 ={\rm Sym}^2{\proj}^1, \quad p_1+p_2\mapsto \pi(p_1)+\pi(p_2). 
\end{equation*}
A general fiber of this fibration can be written as 
$F_1+F_2\simeq F_1\times F_2$ where $F_1\ne F_2$ are general $\pi$-fibers. 
We shall show that the restriction of $\xi^{[2]}$ to $F_1+F_2$ is decomposable. 
We write $\xi$ as 
\begin{equation*}
\xi = \sum_{i}(C_{i}, \phi_{i}) + \sum_{j}(E_{j}, \psi_{j}), 
\end{equation*}
where $C_{i}$ are horizontal curves and $E_{j}$ are vertical curves with respect to $\pi$. 
In other words, $E_{j}$ is contained in a $\pi$-fiber while $\pi(C_{i})={\proj}^1$. 
Since the divisor $E_{j}+S$ of $S^{[2]}$ is disjoint from the surface $F_1+F_2$, 
the vertical components of $\xi^{[2]}$ restrict to zero on $F_1+F_2$. 
On the other hands, as for the horizontal components, we have 
\begin{equation*}
(F_1+F_2) \cap (C_{i}+S) = 
\sum_{p\in C_i\cap F_1} (p+F_2) + \sum_{q\in C_i\cap F_2} (q+F_1).   
\end{equation*} 
If we restrict the induced rational map 
$C_i+S\dashrightarrow C_i \stackrel{\phi_i}{\to} {\proj}^1$ 
to $p+F_2\subset C_i+S$, 
this is constant with value $\phi_i(p)$. 
Similarly for $q+F_1$. 
Hence 
$\xi^{[2]}|_{F_1+F_2}$ 
is decomposable. 
\end{example}

\begin{example}
Let $\pi\colon S\to {\proj}^2$ be a $K3$ surface of degree $2$ and 
$\xi$ be a $(2, 1)$-cycle on $S$. 
Then $\pi$ induces the rational Lagrangian fibration 
\begin{equation*}
S^{[2]} \dashrightarrow ({\proj}^2)^{\vee}, \quad 
p_1+p_2 \mapsto \overline{\pi(p_1) \pi(p_2)},  
\end{equation*}
where $({\proj}^2)^{\vee}$ is the dual plane of ${\proj}^2$. 
The indeterminacy locus is ${\proj}^2\subset S^{[2]}$, the locus parametrizing $\pi$-fibers. 
If $l\subset {\proj}^2$ is a general line, 
the closure of the fiber over $[l]\in ({\proj}^2)^{\vee}$ is ${\rm Sym}^{2}C$ 
where $C=\pi^{-1}(l)$ is the genus $2$ curve over $l$. 
We shall show that the restriction of $\xi^{[2]}$ to this surface is decomposable. 
We write $\xi=\sum_{i}(C_i, \phi_i)$. 
Then 
\begin{equation*}
(C_i+S) \cap {\rm Sym}^2C = 
\sum_{p\in C_i\cap C} (p+C). 
\end{equation*}
If we restrict the induced rational map 
$C_i+S\dashrightarrow C_i \stackrel{\phi_i}{\to} {\proj}^1$ 
to $p+C\subset C_i+S$, 
it is constant with value $\phi_i(p)$. 
Hence $\xi^{[2]}|_{{\rm Sym}^2C}$ is decomposable. 
\end{example}

\begin{example}\label{ex: 3}
Let $Y=Y_{F,t}$ and $\xi_2$ be as in \S \ref{ssec: (2,1)}. 
We take a smooth hyperplane section $W=Y\cap H$ where $H\subset {\proj}^5$ is a hyperplane. 
Then the Fano surface $F(W)$ of the cubic threefold $W$ is a Lagrangian surface of $F(Y)$ 
(\cite{Vo}, see also \cite{Hu} \S 6.4.3). 
We shall show that $\xi_{2}|_{F(W)}$ is decomposable.  
We have 
\begin{equation}\label{eqn: xi2 on F(W)}
\xi_2|_{F(W)} = (Z_1, \psi_1) + (Z_2, \psi_2), 
\end{equation}
where the curve $Z_i\subset F(W)$ is the locus of lines on $W$ intersecting with $W\cap \hat{C}_{i}= H\cap \hat{C}_{i}$, 
and the rational map $\psi_{i}$ is given by 
\begin{equation}\label{eqn: Z_i}
Z_i \dashrightarrow H\cap \hat{C}_{i} \dashrightarrow C_i \stackrel{\phi_i}{\to} {\proj}^1. 
\end{equation}
Here the first map sends a line to its intersection point with $W\cap \hat{C}_{i}$, 
and the second map is the projection from the vertex $p_0$ of $\hat{C}_{i}$. 

We first consider the case the hyperplane $H$ passes through $p_0$. 
In this case $H\cap \hat{C}_{i}$ splits into three lines, say $l_{i1}+l_{i2}+l_{i3}$. 
Then $Z_i$ splits into three components, say $Z_{i1}+Z_{i2}+Z_{i3}$, 
where $Z_{ij}$ is the locus of lines intersecting with $l_{ij}$. 
By our description of $\psi_{i}$, it is constant on each component $Z_{ij}$. 
Therefore $\xi_2|_{F(W)}$ is decomposable. 
\end{example}

\begin{example}
Keeping the notation in Example \ref{ex: 3}, 
we next consider the case $H$ is general and in particular does not pass through $p_0$. 
Then $C_i':=\hat{C}_{i}\cap H$ is a cuspidal plane cubic isomorphic to $C_i$ 
by the projection $\hat{C}_{i}\dashrightarrow C_i$. 
Let $P_i$ be the plane containing $C_i'$; 
this is the intersection of $H$ with the linear hull of $\hat{C}_{i}$. 
We have $W\cap P_i=C_i'$. 
Since $P_1\cap P_2$ is a line (the line joining the three collinear points $C_1'\cap C_2'$), 
the span $\langle P_1, P_2 \rangle$ of $P_1$ and $P_2$ is $3$-dimensional. 
We consider the cubic surface $S=W\cap \langle P_1, P_2 \rangle$. 
We have the rational map 
\begin{equation*}
f : F(W) \dashrightarrow S, \quad [l]\mapsto l\cap \langle P_1, P_2 \rangle. 
\end{equation*}
The indeterminacy locus is $27$ points, the lines on $S$. 
Blowing-up them resolves $f$: 
\begin{equation*}
F(W) \stackrel{g}{\longleftarrow} \widetilde{F(W)} \stackrel{\tilde{f}}{\longrightarrow} S, 
\end{equation*}
where $g$ is the blow-down map and $\tilde{f}$ is a finite morphism (of degree $6$). 

We regard $\phi_i$ as a rational function on $C_i'$ via $C_i'\simeq C_i$ 
and denote it by $\phi_{i}'$. 
Then 
\begin{equation*}
\xi_{S} := (C_1', \phi_1') + (C_2', \phi_2') 
\end{equation*}
is a $(2, 1)$-cycle on $S$. 
By the descriptions \eqref{eqn: xi2 on F(W)} and \eqref{eqn: Z_i}, 
we see that $\xi_{2}|_{F(W)}$ is the pullback of $\xi_S$ by $f$. 
More precisely, we have 
$\xi_{2}|_{F(W)} = g_{\ast}\tilde{f}^{\ast}\xi_{S}$. 
Now, since $S$ is rational, we have ${\CH}^2(S, 1)_{{\rm ind}}=0$. 
Hence $\xi_S$ is indecomposable, and so is $\xi_{2}|_{F(W)}$. 
\end{example}

\begin{remark}\label{rmk: Beilinson Hodge}
A conjecture of de Jeu and Lewis \cite{dJL} predicts that 
the transcendental regulator map 
${\CH}^2(S, 1)_{{\rm ind}} \to J^{2}(S)_{{\rm tr}}$ 
is injective up to torsion for smooth projective surfaces $S$. 
If this is true, 
it follows immediately from Lemma \ref{lem: Lagrangian J2} that 
$\xi|_{S}$ is decomposable up to torsion for Lagrangian surfaces $S$ 
of holomorphic symplectic fourfolds $X$. 
\end{remark}


\end{document}